\documentclass[english,a4paper]{amsart}
\pdfoutput=1

\usepackage[T1]{fontenc}
\usepackage{ae}
\usepackage{aecompl}
\usepackage{amssymb}
\usepackage{mathrsfs}

\usepackage{babel}
\usepackage[leqno]{amsmath}
\usepackage{amsthm}

\newcommand\nos{100} 
\usepackage[svgnames]{xcolor}
\usepackage{tikz}
\usetikzlibrary{shapes,patterns,calc,snakes}

\renewcommand\phi{\varphi}
\renewcommand\epsilon{\varepsilon}
\renewcommand\theta{\vartheta}

\newcommand\mbb{\mathbb}

\newcommand\mcal{\mathscr}

\newcommand\ol{\overline}

\newcommand\ul{\underline}

\newcommand\wt{\widetilde}

\newcommand\sL{\mcal{L}}

\newcommand\sS{\mcal{S}}

\newcommand\fm{\mathfrak{m}}

\newcommand\N{\mbb{N}}

\newcommand\R{\mbb{R}}

\DeclareMathOperator\QM{\rm QM}

\DeclareMathOperator\interior{\rm int}

\newcommand\into{\rightarrow}

\newcommand\isom{\cong}

\renewcommand\le{\leqslant}
\renewcommand\ge{\geqslant}

\numberwithin{equation}{section}

\theoremstyle{plain}
\newtheorem{Thm}[equation]{Theorem}
\newtheorem{Prop}[equation]{Proposition}
\newtheorem{Cor}[equation]{Corollary}
\newtheorem{Lemma}[equation]{Lemma}

\newtheorem*{Thm*}{Theorem}
\newtheorem*{Satz*}{Satz}
\newtheorem*{Prop*}{Proposition}
\newtheorem*{Cor*}{Corollary}
\newtheorem*{Lemma*}{Lemma}
\newtheorem*{Hilfssatz*}{Lemma}
\newtheorem*{Sublemma*}{Sublemma}
\newtheorem*{Conjecture*}{Conjecture}

\theoremstyle{definition}

\newtheorem{Defs}[equation]{Definitions}
\newtheorem{Example}[equation]{Example}

\newtheorem{Remark}[equation]{Remark}
\newtheorem{Remarks}[equation]{Remarks}

\newtheorem*{Def*}{Definition}
\newtheorem*{Defs*}{Definitions}
\newtheorem*{Example*}{Example}
\newtheorem*{Examples*}{Examples}
\newtheorem*{LemmaDef*}{Lemma and Definition}
\newtheorem*{Notation*}{Notation}
\newtheorem*{Problem*}{Problem}
\newtheorem*{Question*}{Question}
\newtheorem*{Remark*}{Remark}
\newtheorem*{Remarks*}{Remarks}
\newtheorem*{Warning*}{Warning}

\newtheorem*{Text*}{}

\renewcommand\nabla\triangledown

\begin{document}
\title[Exposed faces of semidefinitely representable sets]{Exposed faces of\\semidefinitely representable sets}
\author{Tim Netzer}
\address{Fachbereich Mathematik, Universit{\"a}t Konstanz, 78457 Konstanz, Germany}
\email{daniel.plaumann@uni-konstanz.de}
\author{Daniel Plaumann}
\address{Fachbereich Mathematik, Universit{\"a}t Konstanz, 78457 Konstanz, Germany}
\email{tim.netzer@uni-konstanz.de}
\author{Markus Schweighofer}
\address{Universit\'e de Rennes 1\\Laboratoire de Math\'ematiques,
Campus de Beaulieu, 35042 Rennes cedex, France}
\email{markus.schweighofer@univ-rennes1.fr}
\date{15 December 2009}

\begin{abstract}
A linear matrix inequality (LMI) is a condition stating that a
symmetric matrix whose entries are affine linear combinations of
variables is positive semidefinite. Motivated by the fact that
diagonal LMIs define polyhedra, the solution set of an LMI is
called a spectrahedron. Linear images of spectrahedra are called
semidefinitely representable sets. Part of the interest in
spectrahedra and semidefinitely representable sets arises from the
fact that one can efficiently optimize linear functions on them by
semidefinite programming, like one can do on polyhedra by linear
programming.

It is known that every face of a spectrahedron is exposed. This is
also true in the general context of rigidly convex sets. We study
the same question for semidefinitely representable sets. Lasserre
proposed a moment matrix method to construct semidefinite
representations for certain sets. Our main result is that this
method can only work if all faces of the considered set are
exposed. This necessary condition complements sufficient
conditions recently proved by Lasserre, Helton and Nie.
\end{abstract}

\maketitle

\section*{Introduction}A linear matrix polynomial is a symmetric matrix whose entries are real linear
polynomials in $n$ variables. Such a matrix can be evaluated in
any point of $\R^n$, and the set of points where it is positive
semidefinite is a closed convex subset of $\R^n.$ If the matrix is
diagonal, the resulting set is a polyhedron. Since sets defined by
general linear matrix polynomials inherit certain properties from
polyhedra, they are called \emph{spectrahedra}. Sometimes also the
term \emph{LMI (representable) sets} has been used.

Spectrahedra have long been of interest in applications, see for
example the book of Boyd, El Ghaoui, Feron, and Balakrishnan
\cite{MR1284712}. Most importantly, spectrahedra are the feasible
sets of semidefinite programs, which have been much studied in
recent years, as explained for example in Vandenberghe and Boyd
\cite{MR1379041}. Semidefinite programming is a generalization of
linear programming for which there exist efficient algorithms.

Projections of spectrahedra will be called \emph{semidefinitely
representable sets}. They are still useful for optimization.
Indeed, instead of optimizing a linear function on the projection,
one can optimize the same function on the higher dimensional
spectrahedron itself.

In recent years, the fundamental question to characterize
spectrahedra and their projections geometrically has gained a lot
of attention. Helton and Vinnikov have introduced the notion of
rigid convexity, which is an obvious property of spectrahedra.
They show that in dimension two this property characterizes
spectrahedra, and conjecture that the same is true in arbitrary
dimension \cite{MR2292953}. As for semidefinitely representable
sets, the only known property besides convexity is that they are
semialgebraic, i.e. described by a boolean combination of
polynomial inequalities. Indeed, Helton and Nie conjecture that
every convex semialgebraic set is semidefinitely representable
\cite{HeltonNieNecSuffSDP}. Lasserre proposed a construction to
approximate convex semialgebraic sets by semidefinitely
representable sets \cite{MR2505746}. Under certain
conditions this approximation is exact, i.e.~the original set is
semidefinitely representable itself. Helton and Nie have shown that these
conditions are satisfied for a surprisingly large class of sets, see
\cite{MR2292953} Theorem 5.1. They also prove that 
Lasserre's method can be applied locally for compact sets. This
allows them to show semidefinite representability for an even
larger class of sets.

In this work, we investigate the facial geometry of spectrahedra,
rigidly convex sets and semidefinitely representable sets. It is
known that all faces of a spectrahedron are exposed. We review
this fact in Section \ref{seczwei} and prove the same for
rigidly convex sets, as a consequence of Renegar's result for hyperbolicity cones \cite{MR2198215}. Our main result is Theorem \ref{main} in Section \ref{secdrei}. We prove that Lasserre's construction can
only be exact if all faces of the considered convex set are
exposed. This is a necessary condition which complements the
sufficient conditions from the above mentioned literature. We use real algebra, basic model theory, and convex geometry in our
proof.

\section{Preliminaries}

Let $\R[\ul t]$ denote the polynomial ring in $n$ variables
$\ul t=(t_1,\dots,t_n)$ with coefficients in $\R$. A subset $S$ of
$\R^n$ is called \emph{basic closed} if there exist polynomials
$p_1,\dots,p_m\in\R[\ul t]$ such that
\[
S=\sS(p_1,\dots,p_m)=\bigl\{x\in\R^n\:\bigl|\: p_1(x)\ge 0,\dots,p_m(x)\ge 0\bigr\}.
\]
A \emph{linear matrix polynomial} (of dimension $k$ in the
variables $\ul t$) is a linear polynomial whose coefficients are
real symmetric $k\times k$-matrices, i.e.~an expression $A(\ul
t)=A_0+t_1A_1+\cdots +t_nA_n$ with $A_0,\dots,A_n\in {\rm
Sym_k(\R)}$.  A subset $S$ of $\R^n$ is called a
\emph{spectrahedron}, if it is defined by a linear matrix
inequality, i.e. if there exists a linear matrix polynomial $A(\ul
t)$ such that
\[
S=\sS(A)=\bigl\{x\in\R^n\:\bigl|\: A(x)=A_0+x_1A_1+\cdots+x_nA_n\succeq 0\bigr\},
\]
where $\succeq 0$ denotes positive semidefiniteness. It is obvious
that spectrahedra are closed and convex. They are also basic
closed: A real symmetric matrix is positive semidefinite if and
only if the coefficients of its characteristic polynomial have
alternating signs; write
\[
\det(A(\ul t)-sI_k)=c_0(\ul t)+c_1(\ul t)s+\cdots+c_{k-1}(\ul t)s^{k-1}+(-1)^ks^k
\]
with $p_i\in\R[\ul t]$, then
\[
\sS(A)=\sS(c_0,-c_1,\dots,(-1)^{k-1}c_{k-1}).
\]
A further property of spectrahedra is their rigid convexity: A
polynomial $p\in\R[\ul t]$ is called a \emph{real zero polynomial
 with respect to $e\in\R^n$ (RZ$_e$-polynomial)} if $p(e)>0$ and all
zeros of the univariate polynomial $p(e+sv)\in\R[s]$ are real, for
every $v\in\R^n\setminus \{0\}$.  A set $S$ is called \emph{rigidly
 convex} if there exists $e\in S$ and an RZ$_e$-polynomial $p$ such
that $S$ is the closure of the connected component of
$\{x\in\R^n\:|\: p(x)>0\}$ containing $e$. Rigid convexity was
introduced and studied by Helton and Vinnikov \cite{MR2292953}.
Rigidly convex sets are convex (see Section 5.3 in
\cite{MR2292953}); they are also basic closed (see Remark
\ref{Rem:RigidConvexBasicClosed} below). Furthermore, any
spectrahedron with non-empty interior is rigidly convex. The
principal reason is that if $A(\ul t)$ is a linear matrix
polynomial with $A_0\succ 0$, then $p(\ul t)=\det(A(\ul t))$ is an
RZ$_0$-polynomial defining $\sS(A)$ (see \cite{MR2292953},
Thm.~2.2). A much harder question is whether every rigidly convex
set is a spectrahedron. This has been shown for $n=2$ and
conjectured in general by Helton and Vinnikov in \cite{MR2292953}.
The question is closely related to the famous Lax-conjecture.

A subset $S$ of $\R^n$ is called \emph{semidefinitely representable}
if it is the image of a spectrahedron $S'$ in $\R^m$ under a
linear map $\R^m\into\R^n$. A linear matrix representation of $S'$
together with the linear map is called a \emph{semidefinite
representation} of $S$. In contrast to spectrahedra, no necessary
conditions other than convexity are known for a semialgebraic set
to be semidefinitely representable.

Various sufficient conditions
have recently been given by Lasserre \cite{MR2505746} as well as
Helton and Nie \cite{MR2533752},
\cite{HeltonNieNecSuffSDP}. Moreover, it has been shown that various operations, like
taking the interior or taking the convex hull of a finite union, preserve semidefinite representability, see \cite{TimRainer} and \cite{Tim}.

\section{Faces of spectrahedra and rigidly convex
sets}\label{seczwei}

In this section, we study the facial structure of spectrahedra and
rigidly convex sets (see also \cite{MR2322886} for a discussion of facial structures
in a more abstract setting). We review the result of Ramana and Goldman
that every spectrahedron has only exposed faces. We then discuss
how the same result can be proven for rigidly convex sets, mostly
by going back to Renegar's corresponding result for hyperbolicity
cones.

\begin{Defs}
Let $S$ be a closed convex subset of $\R^n$ with non-empty
interior. A \emph{supporting hyperplane} of $S$ is an affine
hyperplane $H$ in $\R^n$ such that $S\cap H\neq\emptyset$ and
$S\setminus H$ is connected (equivalently, the zero set of a
linear polynomial $0\neq \ell\in\R[\ul t]$ such that $\ell\ge 0$
on $S$ and $\{\ell=0\}\cap S\neq\emptyset$).

A \emph{face} of $S$ is a non-empty convex subset $F\subseteq S$ with
the following property: For every $x,y\in S$, $\lambda\in (0,1)$, if
$\lambda x+(1-\lambda) y\in F$, then $x,y\in F$.

A face $F$ of $S$ is called \emph{exposed} if either $F=S$ or there exists a
supporting hyperplane $H$ of $S$ such that $H\cap S=F$. The hyperplane
$H$ is said to \emph{expose} $F$.

The \emph{dimension} of a face $F$ is the dimension of its affine
hull.
\end{Defs}

\begin{Remarks}\item \label{Remark:Faces}
\begin{enumerate}
\item $H\cap S$ is an exposed face of $S$ for any supporting
hyperplane $H$ of $S$.

\item For every face $F\subsetneq S$ there exists a supporting
hyperplane $H$ of $S$ such that $F\subseteq H$.

\item Every face of $S$ is closed (since $S$ is closed).

\item If $F_1,F_2$ are faces of $S$ with $F_1\subsetneq F_2$, then
$\dim(F_1)<\dim(F_2)$.

\item Let $F$ be a face of $S$, and take $x_0$ in the relative
interior of $F$. For any two points $x\neq y\in\R^n$, let $g(x,y)$
denote the line passing through $x$ and $y$. Then $F$ consists
exactly of $x_0$ and  those points $x\in S\setminus\{x_0\}$ such
that $x_0$ lies in the relative interior of $g(x,x_0)\cap S$.
\end{enumerate}
\end{Remarks}

The following is a combination of Theorem~1 and Corollary~1 in
\cite{MR1342934} (see Corollary~1 in \cite{MR2322886} for a more general statement).
\begin{Thm}[Ramana and Goldman]
Let $A(\ul t)$ be a linear matrix polynomial of dimension $k$, $S=\sS(A)$.
For every linear subspace $U$ of $\R^k$, the set
\[
F_U=\bigl\{x\in S\:|\: U\subseteq \ker\bigl(A(x)\bigr)\bigr\}
\]
is a face of $S$ or empty, and every face of $S$ is of this form.
Furthermore, every face of $S$ is exposed.
\end{Thm}

A similar result can be proven for rigidly convex sets, by
reducing to the results of Renegar on hyperbolicity cones that we
now describe: A homogeneous polynomial $P$ in $n+1$ variables is
called \emph{hyperbolic with respect to
$e\in\R^{n+1}\setminus\{0\}$} if $P(e)>0$ and all zeros of the
univariate polynomial $P(x-se)\in\R[s]$ are real, for every
$x\in\R^{n+1}$. The \emph{hyperbolicity cone of $P$} is the
connected component of $\{P>0\}$ containing $e$. It is a convex
cone in $\R^{n+1}$. Its closure is called the \emph{closed
hyperbolicity cone of $P$}.

\begin{Thm}[Renegar \cite{MR2198215}, Thm.~23]\label{Thm:RenegarExposedFaces}
The faces of a closed hyperbolicity cone are exposed.
\end{Thm}

\begin{Cor}
The faces of a rigidly convex set are exposed.
\end{Cor}

\begin{proof}
It is well-known and easy to see that a polynomial $p\in\R[\ul t]$ is
an $RZ_e$-polynomial if and only if the homogenisation $P(\ul
t,u)=u^dp(\frac{\ul t}{u})$ is hyperbolic with respect to $\wt
e=(e,1)$. Furthermore, the rigidly convex set $S\subseteq\R^n$ defined
by $p$ (i.e.~the closure of the connected component of $\{p>0\}$
containing $e$) is the intersection of $C$, the closed hyperbolicity
cone of $P$ in $\R^{n+1}$, with the hyperplane $H=\{u=1\}$.

Let $F_0$ be a face of $S$. For any two points $x\neq
y\in\R^{n+1}$, let $g(x,y)$ denote the line passing through $x$
and $y$. Take $x_0$ in the relative interior of $F_0$, and let $F$
be the set of all points $z\in C$ such that $x_0$ lies in the
relative interior of $g(z,x_0)\cap C$. One checks that $F$ is a
face of $C$ and that $F\cap H=F_0$ (see Remark \ref{Remark:Faces}
(5)). Since $F$ is exposed by Thm.~\ref{Thm:RenegarExposedFaces},
so is $F_0$.
\end{proof}

The idea of the proof of Renegar's theorem is the following: Let $P$ be a
homogeneous polynomial in $n+1$ variables $(\ul t,u)$ that is
hyperbolic with respect to $e\in\R^{n+1}\setminus\{0\}$, and let $C$
be the closed hyperbolicity cone of $P$. For every $k\ge 0$, put
\[
P^{(k)}(\ul t,u)=\frac {d^k}{ds^k} P\bigl((\ul t,u)+se\bigr)\biggl|_{s=0}.
\]
The polynomials $P^{(k)}$ are again hyperbolic with respect to $e$ (by Rolle's
theorem) and the corresponding closed hyperbolicity cones $C^{(k)}$ form an
ascending chain $C=C^{(0)}\subseteq C^{(1)}\subseteq C^{(2)}\subseteq\cdots$.
For $x\in C$, define ${\rm mult}(x)$ as the multiplicity of $0$
as a zero of the univariate polynomial $P(x+se)\in\R[s]$. If ${\rm
 mult}(x)=m$, then $x$ is a boundary point of $C^{(m-1)}$ and a regular
point of $\{P^{(m-1)}=0\}$, i.e.~$(\nabla P^{(m-1)})(x)\neq
0$. Now if $F$ is a face of $C$ and $x$ is in the relative interior of
$F$, then the tangent space of $P^{(m-1)}$ in $x$ exposes $F$ as a face
of $C^{(m-1)}$ and hence as a face of $C$.

This translates into the setting of rigid convexity as follows: Let
$p\in\R[\ul t]$ be an $RZ_0$-polynomial of degree $d$, and let $S$ be
the corresponding rigidly convex set; write $p=\sum_{i=0}^d p_i$ with
$p_i$ homogeneous of degree $i$, and put $P(\ul t,u)=u^dp(\frac {\ul
 t}{u})=\sum_{i=0}^d p_{d-i}(\ul t)u^i$. Define $P^{(k)}$ for $k\ge
0$ as above and put $p^{(k)}(\ul t)=P^{(k)}(\ul t,1)$, so that
\[
p^{(k)}(\ul t)=\sum_{i=k}^d \frac{i!}{(i-k)!} p_{d-i}(\ul t).
\]
The polynomials $p^{(k)}$ are again RZ$_0$-polynomials and the
corresponding rigidly convex sets form an ascending chain
$S=S^{(0)}\subseteq S^{(1)}\subseteq S^{(2)}\subseteq\cdots$. For any
$x\in S$, we find that ${\rm mult}(x)$ is the multiplicity of $0$ as a
zero of the univariate polynomial $\sum_{i=0}^d
p_{d-i}(x)(1+s)^i\in\R[s]$. A simple computation shows that ${\rm
 mult}(x)$ is also the multiplicity of $1$ as a zero of the
univariate polynomial $p(sx)\in\R[s]$.

Now let $F$ be a face of $S$, let $x$ be a point in the relative
interior of $F$, and put $m={\rm mult}(x)$. Then $x$ is a boundary point of
$S^{(m-1)}$ and a regular point of $\{p^{(m-1)}=0\}$. The tangent
space $\{x+v\:|\: (\nabla p^{(m-1)}(x))^tv=0\}$ exposes $F$
as a face of $S$.

\begin{Remark}\label{Rem:RigidConvexBasicClosed}
It follows from Renegar's construction that closed hyperbolicity
cones and rigidly convex sets are basic closed semialgebraic sets.
Namely, if $C$ is the closed hyperbolicity cone of a hyperbolic
polynomial $P$ of degree $d$, then
$C=\sS(P,P^{(1)},\dots,P^{(d-1)})$; similarly, if $S$ is a rigidly
convex set corresponding to an RZ$_0$-polynomial $p$ of degree
$d$, then $S=\sS(p,p^{(1)},\dots,p^{(d-1)})$.

Alternatively, one can use the fact that the closed hyperbolicity cone of
$P$ coincides with the set of all $ x\in\R^{n+1}$ such that all
zeros of $P(x-se)\in\R[s]$ are nonnegative. This translates to an
alternating sign condition on the coefficients with respect to
$s$, as explained in Section 1.

\end{Remark}

\begin{Example}
Let $p=t_1^3-t_1^2-t_1-t_2^2+1\in\R[t_1,t_2]$. One checks
that $p$ is an irreducible RZ$_0$-polynomial. The corresponding
rigidly convex set, i.e.~the closure of the connected component of $\{p>0\}$
containing $0$, is the basic closed set $S=\sS(p,1-t_1)$.

We have ${\rm mult}(x)=1$ for every boundary point $x\in\partial
S\setminus\{(1,0)\}$, and ${\rm mult}(1,0)=2$. Furthermore,
$p^{(1)}=-t_1^2-t_2^2-2t_1+3$, $p^{(2)}=6-t_1$. Every
$x\in\partial S\setminus\{(1,0)\}$ is a regular point of $\{p=0\}$
and is exposed as a face of $S$ by the tangent line to $\{p=0\}$
in $x$. The point $(1,0)$ is a regular point of $\{p^{(1)}=0\}$
and is exposed as a face of $S$ by the tangent line to that curve
in $(1,0)$, which is $t_1=1$. We also see that
$S=\sS(p,p^{(1)},p^{(2)})$ (though $p^{(2)}$ is redundant):
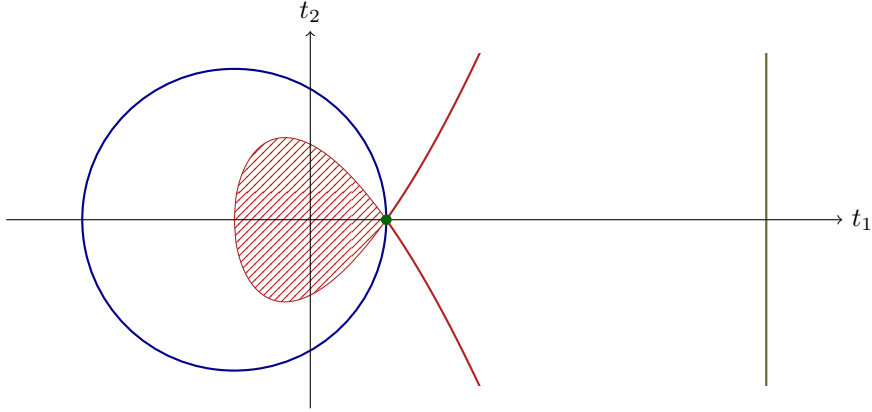
\begin{center}
\begin{tikzpicture}
\begin{scope}
\clip (-4,-2.2) rectangle (7,2.2);
\pgfsetstrokecolor{FireBrick};
\pgfsetfillpattern{north east lines}{FireBrick};
\filldraw[smooth,domain=0:2,samples=\nos] plot({\x-1},{0-(sqrt((\x-1)^3-(\x-1)^2-(\x-1)+1))});
\draw[smooth,domain=2:4,thick] plot({\x-1},{0-(sqrt((\x-1)^3-(\x-1)^2-(\x-1)+1))});
\filldraw[smooth,domain=0:2,samples=\nos] plot({\x-1},{sqrt((\x-1)^3-(\x-1)^2-(\x-1)+1)});
\draw[smooth,domain=2:4,thick] plot({\x-1},{sqrt((\x-1)^3-(\x-1)^2-(\x-1)+1)});
\pgfsetstrokecolor{DarkBlue};
\draw[thick] (-1,0) circle (2);
\pgfsetstrokecolor{DarkOliveGreen};
\draw[thick] (6,-4) -- (6,4);
\end{scope}
\draw[->] (-4,0) -- (7,0) node[right]{$t_1$};
\draw[->] (0,-2.5) -- (0,2.5) node[above]{$t_2$};
\fill[color=DarkGreen] (1,0) circle (2pt);
\end{tikzpicture}
\end{center}

By the theorem of Helton and Vinnikov, $S$ is a spectrahedron.
Explicitly, let $A(t_1,t_2)=A_0+t_1A_1+t_2A_2$ with
\[
A_0=
\left(
\begin{array}{ccc}
2 & 0 & 1\\
0 & 1 & 0\\
1 & 0 & 1
\end{array}
\right),
\quad
A_1=
\left(
\begin{array}{ccc}
-2 & 0 & -1\\
0 & -1 & 0\\
-1 & 0 & 0
\end{array}
\right),
\quad
A_2=
\left(
\begin{array}{ccc}
0 & 1 & 0\\
1 & 0 & 0\\
0 & 0 & 0
\end{array}
\right).
\]
For the characteristic polynomial, one finds
$\chi_A(s)=c_0+c_1s+c_2s^2-s^3$ with $c_0=p$,
$c_1=-t_1^2+5t_1+t_2^2-4$, $c_2=4-3t_1$. One checks that
$S=\sS(A)=\sS(c_0,-c_1,c_2)=\sS(c_0,-c_1)$. This gives an alternative
description of $S$ as a basic closed set.
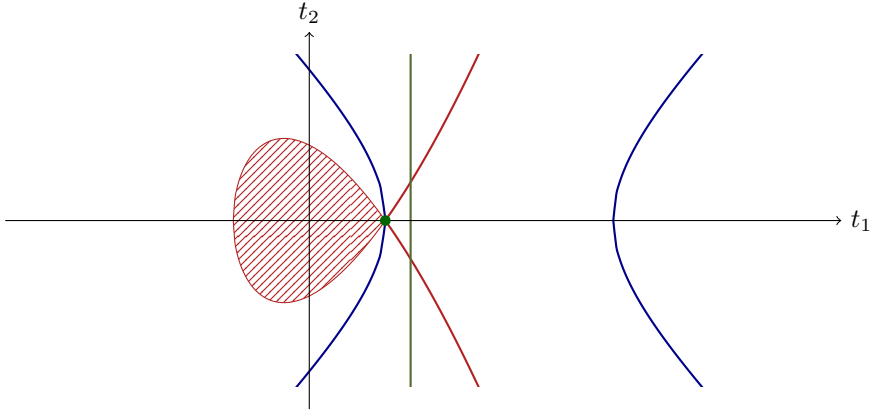
\begin{center}
\begin{tikzpicture}
\begin{scope}
\clip (-4,-2.2) rectangle (7,2.2);
\pgfsetstrokecolor{FireBrick};
\pgfsetfillpattern{north east lines}{FireBrick};
\filldraw[smooth,domain=0:2,samples=\nos] plot({\x-1},{0-(sqrt((\x-1)^3-(\x-1)^2-(\x-1)+1))});
\draw[smooth,domain=2:4,thick] plot({\x-1},{0-(sqrt((\x-1)^3-(\x-1)^2-(\x-1)+1))});
\filldraw[smooth,domain=0:2,samples=\nos] plot({\x-1},{sqrt((\x-1)^3-(\x-1)^2-(\x-1)+1)});
\draw[smooth,domain=2:4,thick] plot({\x-1},{sqrt((\x-1)^3-(\x-1)^2-(\x-1)+1)});
\pgfsetstrokecolor{DarkBlue};
\draw[smooth,domain=-5:1,samples=\nos,thick] plot({\x},{sqrt(\x^2-5*\x+4)});
\draw[smooth,domain=4:8,samples=\nos,thick] plot({\x},{sqrt(\x^2-5*\x+4)});
\draw[smooth,domain=-5:1,samples=\nos,thick] plot({\x},{0-sqrt(\x^2-5*\x+4)});
\draw[smooth,domain=4:8,samples=\nos,thick] plot({\x},{0-sqrt(\x^2-5*\x+4)});
\pgfsetstrokecolor{DarkOliveGreen};
\draw[thick] (4/3,-4) -- (4/3,4);
\end{scope}
\draw[->] (-4,0) -- (7,0) node[right]{$t_1$};
\draw[->] (0,-2.5) -- (0,2.5) node[above]{$t_2$};
\fill[color=DarkGreen] (1,0) circle (2pt);
\end{tikzpicture}
\end{center}
\end{Example}

\section{Exposed faces and Lasserre relaxations}\label{secdrei}

For a certain class of convex semialgebraic sets, Lasserre has
given an explicit semidefinite representation
\cite{MR2505746} (see \cite{MR1940975} for a less well known but related construction), as follows: Let $\ul p=(p_1,\dots,p_m)$
be an $m$-tuple of real polynomials in $n$ variables $\ul t$, and
set $p_0=1$. Let $\QM(\ul p)$ be the quadratic module generated by
$\ul p$, i.e.~
\[
\QM(\ul p)=\left\{\sum_{i=0}^m \sigma_i p_i\:\bigl|\:
\sigma_i\in\sum\R[\ul t]^2\right\}
\]
where $\sum\R[\ul t]^2=\{f_1^2+\cdots +f_r^2\:|r\ge 0,
f_1,\dots,f_r\in\R[\ul t]\}$. We denote by $\R[\ul t]_d$ the
finite-dimensional vector space of polynomials of degree at most $d$,
and write $\R[\ul t]_d^\vee$ for its (algebraic) dual. Define
\[
\QM(\ul p)_d=\left\{\sum_{i=0}^m \sigma_i p_i\:\bigl|\:
\sigma_i\in\sum\R[\ul t]^2;\: \sigma_ip_i\in\R[\ul t]_d\right\}.
\]
Note that the inclusion $\QM(\ul p)_d\subseteq\QM(\ul p)\cap\R[\ul
t]_d$ is in general not an equality. Let
\[
\sL(\ul p)_d=\bigl\{L\in\R[\ul t]_d^\vee\:\bigl|\: L|_{\QM(\ul
 p)_d}\ge 0, L(1)=1\bigr\}.
\]
It is well-known that $\sL(\ul p)_d$ is a spectrahedron in $\R[\ul
t]_d^\vee$ (see for example Marshall \cite{MR2383959}, 10.5.4).
Now consider the projection $\pi\colon\R[\ul t]_d^\vee\into\R^n$,
$L\mapsto (L(t_1),\dots,L(t_n))$ and put
\[
S(\ul p)_d=\pi\bigl(\sL(\ul p)_d\bigr),
\]
a semidefinitely representable subset of $\R^n$. The idea is to
compare $S(\ul p)_d$ with $S=\sS(\ul p)$, the basic closed set
determined by $\ul p$. Note first that $S(\ul p)_d$ contains $S$
and therefore its convex hull: For if $x\in S$, let
$L_{x}\in\R[\ul t]_d^\vee$ denote evaluation in $x$; then
$L_x\in\sL(\ul p)_d$ and $\pi(L_x)=x$. Note also that the sets
$S(\ul p)_d$ form a decreasing sequence, i.e.
$$S(\ul p)_{d+1}\subseteq S(\ul p)_d$$ holds for all $d$.

\medskip We call the set $S(\ul p)_d$ the $d$-th
\emph{Lasserre relaxation} of ${\rm conv}(S)$ with respect to $\ul
p$. If there exists $d\ge 0$ such that $S(\ul p)_d={\rm conv}(S)$,
we say that ${\rm conv}(S)$ possesses an \emph{exact Lasserre
relaxation} with respect to $\ul p$. The existence of an exact
Lasserre relaxation is a sufficient condition for the semidefinite
representability of ${\rm conv}(S)$.

A characterization for exactness of Lasserre relaxations is the
following proposition. The implication (2)$\Rightarrow$(1) is
\cite{MR2505746}, Thm.~2.

\begin{Prop}\label{Lasserre:Exact}
Assume that $S=\sS(\ul p)$ has non-empty interior. For $d\in\N$,
the following are equivalent:
\begin{enumerate}
\item ${\rm conv}(S)\subseteq S(\ul p)_d\subseteq \ol{{\rm conv}(S)}$;
\item Every $\ell\in\R[\ul t]_1$ with $\ell|_S\ge 0$ is contained in
$\QM(\ul p)_d$.
\end{enumerate}
\end{Prop}

\begin{proof}
We include the proof of (2)$\Rightarrow$(1) for the sake of
completeness. So assume that (2) holds and suppose that there
exists $x\in S(\ul p)_d\setminus \ol{{\rm
   conv}(S)}$. Thus there is $\ell\in\R[\ul t]_1$ with $\ell|_S\ge 0$ and
$\ell(x)<0$. Furthermore, there exists a linear functional $L\colon\R[\ul
t]_d\into\R$ such that $L|_{\QM(\ul p)_d}\ge 0$, $L(1)=1$, and
$x=\bigl(L(t_1),\dots,L(t_n)\bigr)$. By assumption, $\ell$ belongs to
$\QM(\ul p)_d$, so $0\le L(\ell)=\ell(L(t_1),\dots,L(t_n))=\ell(x)<0$, a
contradiction.

For the converse, assume that (1) holds, and suppose that there
exists $\ell\in\R[\ul t]_1$ with $\ell|_S\ge 0$ but
$\ell\notin\QM(\ul p)_d$. Since $S$ has non-empty interior,
$\QM(\ul p)_d$ is a closed convex cone in $\R[\ul t]_d$ (see for
example Marshall \cite{MR2383959}, Lemma 4.1.4, or Powers and
Scheiderer \cite{MR1823953}, Proposition 2.6). Thus there exists a linear
functional $L\colon\R[\ul t]_d\into\R$ such that $L|_{\QM(\ul
p)_d}\ge 0$, $L(1)=1$, and $L(\ell)<0$ (note that $L(1)=1$ is
non-restrictive; see the little trick in Marshall \cite{MR2011395}, proof
of Theorem 3.1). Since $x=(L(t_1),\dots,L(t_n))\in S(\ul
p)_d\subseteq\ol{{\rm conv}(S)}$, we have $0\le
\ell(x)=L(\ell)<0$, a contradiction.
\end{proof}

An immediate consequence is that if ${\rm conv}(S)$ is closed (for
example if $S$ is compact or convex), then (2) implies that ${\rm
 conv}(S)$ is semidefinitely representable. Lasserre shows that (2) is satisfied
for certain classes of sets, for example if all $p_i$ are linear
or concave and quadratic. These results have been extended
substantially  by Helton and Nie
\cite{MR2533752,HeltonNieNecSuffSDP}.

\medskip In the following, we will give a necessary condition for (2)
in the case that $S$ is convex. Namely, all faces of $S$ must be
exposed. The following lemma and its proof are a special case of
Prop.~II.5.16 in Alfsen \cite{MR0445271}.

\begin{Lemma}
Let $S$ be a closed convex subset of $\R^n$. A face $F$ of $S$ is
exposed if and only if for every $x\in S\setminus F$ there exists a
supporting hyperplane $H$ of $S$ with $F\subseteq H$ and $x\notin H$.
\end{Lemma}

\begin{proof}
Necessity is obvious. To prove sufficiency, write $F=\bigcap_{k\ge 1}
U_k$ with $U_k$ open subsets of $\R^n$ such that $\R^n\setminus U_k$
is compact for every $k\ge 1$ (note that $F$ is closed by Remark
\ref{Remark:Faces} (3)). Fix $k\ge 1$. For each $x\in S\setminus
U_k$, we can choose by hypothesis a linear polynomial $\ell_x\in\R[\ul
t]$ such that $\{\ell_x=0\}$ is a supporting hyperplane of $S$ with
$\ell_x|_F=0$ und $\ell_x(x)>0$. Since $S\setminus U_k$ is compact, we may
choose $x_1,\dots,x_m\in S\setminus U_k$ such that $\ell_k:=\sum_{i=1}^m
\ell_{x_i}$ is strictly positive on $S\setminus U_k$. Clearly,
$\ell_k|_F=0$. Put
\[
\ell:=\sum_{k=1}^\infty \frac{\ell_k}{2^k\cdot ||\ell_k||},
\]
where $||\cdot||$ is a norm on the space of linear polynomials. Then
$\{\ell=0\}$ is a supporting hyperplane of $S$ that exposes $F$.
\end{proof}

\begin{Lemma}\label{Lemma:ExposedDimReduction}
Let $S$ be a closed convex subset of $\R^n$ with non-empty interior. A
face $F$ of $S$ is exposed if and only if $F\cap U$ is an exposed face
of $S\cap U$ for every affine-linear subspace $U$ of $\R^n$ containing
$F$ with $\dim(U)=\dim(F)+2$ and $U\cap\interior(S)\neq\emptyset$.
\end{Lemma}

\begin{proof}
Note first that the condition is empty if $F$ is of dimension $\ge
n-1$. Indeed, $F$ is always exposed in that case by Remark
\ref{Remark:Faces} (2),(4). Thus we may assume
that $n\ge 2$ and $\dim(F)\le n-2$.

If $H$ exposes $F$ and $U\cap\interior(S)$ is non-empty, then $H\cap
U$ exposes $F$ in $S\cap U$. Conversely, assume that $F\cap U$ is an
exposed face of $S\cap U$ for every $U$ satisfying the hypotheses. We
want to apply the preceding lemma. Let $x\in S\setminus F$, then we
must produce a supporting hyperplane $H$ of $S$ containing $F$ with
$x\notin H$. Choose $U$ to be an affine-linear subspace of $\R^n$ of
dimension $\dim(F)+2$ containing $F$ such that $x\in U$ and
$U\cap\interior(S)\neq\emptyset$. By hypothesis, there exists a
supporting hyperplane $G$ of $S\cap U$ in $U$ that exposes $F$ as a
face of $S\cap U$. In particular, $x\notin G$. Since $G\cap S=F$, it
follows that $G\cap\interior(S)=\emptyset$, hence by separation of
disjoint convex sets (see e.g.~Barvinok \cite{MR1940576},
Thm.~III.1.2), there exists a hyperplane $H$ that satisfies
$G\subseteq H$ and $H\cap \interior(S)=\emptyset$. Since
$U\cap\interior(S)\neq\emptyset$, it follows that $G\subseteq H\cap
U\subsetneq U$, hence $G=H\cap U$. Thus $H$ is a supporting hyperplane
of $S$ containing $F$ with $x\notin H$.
\end{proof}

We need the following technical lemma.

\begin{Lemma}\label{Lemma:ExtSep}
Let $S$ be a convex subset and $U$ be an affine-linear subspace of $\R^n$ intersecting
the interior of $S$.
Suppose that $\ell\colon\R^n\to\R$ is an affine linear function such that
$\ell\ge 0$ on $S\cap U$. Then there exists an affine linear function
$\ell'\colon\R^n\to\R$ such that $\ell'\ge 0$ on $S$ and $\ell'|_U=\ell|_U$.
\end{Lemma}

\begin{proof}
Let $N:=\{x\in U\mid\ell(x)<0\}$ and $S'$ be the convex hull of
$\{x\in U\mid\ell(x)\ge 0\}\cup S$. Then $N$ and $S'$ are convex sets
that we now prove to be disjoint.

Assume for a contradiction that
there are $\lambda\in[0,1]$, $x\in U$ and $y\in S$ such that $\ell(x)\ge0$ and
$\lambda x+(1-\lambda)y\in N$. Since neither $x$ nor $y$ lies in $N$, we have
$\lambda\not\in\{0,1\}$. Since $U$ is an affine linear subspace,
$\lambda x+(1-\lambda)y\in U$ now implies $y\in U$ and therefore $\ell(y)\ge 0$, leading
to the contradiction
$0>\ell(\lambda x+(1-\lambda)y)=\lambda\ell(x)+(1-\lambda)\ell(y)\ge 0$.

Without loss of generality $N\neq\emptyset$ (otherwise $\ell|_U=0$ and we can take
$\ell'=0$). Then by separation of non-empty disjoint convex sets
(e.g., Thm.~III.1.2 in Barvinok \cite{MR1940576}), we get an affine linear
$\ell'\colon\R^n\to\R$, not identically zero, such that
$\ell'\ge 0$ on $S'$ and $\ell'\le 0$ on $N$. In particular, $\ell'\ge0$ on $S$ and
$\ell'$ cannot vanish at an interior point of $S$.
Since $U$ intersects by hypothesis the interior of $S$, it is not possible that $\ell'$
vanishes identically on $U$. Moreover, all $x\in U$ with $\ell(x)=0$ lie at the same
time in $S'$ and in the closure of $N$, implying that $\ell'(x)=0$. This shows that
the restrictions of $\ell$ and $\ell'$ on $U$ are the same up to a positive factor
which we may assume to be $1$ after rescaling.
\end{proof}

\noindent We are now ready for the main result:

\begin{Thm}\label{main}
Let $S=\sS(\ul p)$ be a basic closed convex subset of $\R^n$ with
non-empty interior. Suppose that there exists $d\ge 1$ such that
the $d$-th Lasserre relaxation of $S$ with respect to $\ul p$ is
exact, i.e.
$$S(\ul p)_d=S$$ holds. Then all faces of $S$ are exposed.
\end{Thm}

In view of Proposition \ref{Lasserre:Exact}, we have the following
equivalent formulation of the same theorem:
\begin{Thm*}[Alternative formulation]
Let $S=\sS(\ul p)$ be a basic closed convex subset of $\R^n$ with
non-empty interior. Suppose that there exists $d\ge 1$ such that
every linear polynomial $\ell$ with $\ell\ge 0$ on $S$ is
contained in $\QM(\ul p)_d$. Then all faces of $S$ are exposed.
\end{Thm*}

\begin{proof}
We begin by showing that it is sufficient to prove that all faces of
dimension $n-2$ are exposed. Let $F$ be a face of $S$ of dimension
$e$. For $e\ge n-1$ there is nothing to show, so assume $e\le n-2$. If
$F$ is not exposed, then by Lemma \ref{Lemma:ExposedDimReduction}
there exists an affine-linear subspace $U$ of $\R^n$ containing $F$
with $\dim(U)=e+2$ and $U\cap\interior(S)\neq\emptyset$ and such that
$F$ is a non-exposed face of $S\cap U$. Furthermore, by Lemma
\ref{Lemma:ExtSep}, for every linear
polynomial $\ell$ that is psd on $S\cap U$ there exists a linear
polynomial $\ell'$ that is psd on $S$ and agrees with $\ell$ on $U$.
Upon replacing $\R^n$ by $U$ and $S$
by $S\cap U$, we reduce to the case $e=n-2$.

Now assume for contradiction that $d\ge 1$ as in the statement exists
and that $F$ is a face of dimension $n-2$ that is not exposed.

\emph{Step 1.} There is exactly one supporting hyperplane $H$ of $S$
that contains $F$. For if $\ell_1,\ell_2$ are non-zero linear polynomials with
$\ell_i|_F=0$ and $\ell_i|_S\ge 0$, put $W:=\{\ell_1=0\}\cap\{\ell_2=0\}$. Then
$\ell:=\ell_1+\ell_2$ defines a supporting hyperplane $\{\ell=0\}$ of $S$ with
$\{\ell=0\}\cap S=W\cap S$. If $\ell_1,\ell_2$ are linearly independent, then
$\dim(W)=n-2=\dim(F)$, hence $F=\{\ell=0\}\cap S$, contradicting the fact
that $F$ is not exposed.

\smallskip We may assume after an affine change of coordinates that
$H=\{t_1=0\}$, $t_1\ge 0$ on $S$, and that $0$ lies in the relative
interior of $F$. Note that any supporting hyperplane of $S$ containing
$0$ must contain $F$ and therefore coincide with $H$.

Since $F$ is not exposed, $F_0=H\cap S$ is a face of dimension $n-1$
with $F$ contained in its relative boundary. In particular, it follows
that $F$ is also contained in the closure of $\partial S\setminus H$.

\medskip

\emph{Step 2.} By the curve selection lemma (see e.g.~Thm.~2.5.5.~in
Bochnak, Coste, and Roy \cite{MR1659509}), we may choose a continuous
semialgebraic path $\gamma\colon [0,1]\into\partial S$ such that
$\gamma(0)=0\in F$, $\gamma\bigl((0,1]\bigr)\cap H=\emptyset$. We
relabel $p_0,\dots,p_m$ into two groups $f_1,\dots,f_r$,
$g_1,\dots,g_s$ as follows:

\begin{center}
\begin{tabular}{ll}
$f_i|_{\gamma([0,1])}=0$ & ($i=1,\dots,r$)\\
$g_j|_{\gamma((0,1])}>0$ & ($j=1,\dots,s$)\\

\end{tabular}
\end{center}
(Indeed, after restricting $\gamma$ to $[0,\alpha]$ for suitable
$\alpha\in (0,1]$ and reparametrizing, we can assume that each $p_i$
falls into one of the above categories.)

We claim that there exists an expression
\[
t_1=\sum_{i=1}^r \rho_i f_i + \sum_{j=1}^s \sigma_j g_j
\leqno{(\ast)}
\]
with $\rho_i,\sigma_j\in\sum\R[\ul t]^2$ and such that
$\sigma_j(0)=0$ for all $j=1,\dots,s$.

To prove the existence of the expression $(\ast)$, consider the
following statement:
\begin{itemize}
\item[($\dagger$)] For each $\lambda\in (0,1]$ there exists a
linear polynomial $\ell_\lambda\in\R[\ul t]_1$ such that
$\ell_\lambda(\gamma(\lambda))=0$, $\ell_\lambda\ge 0$ on $S$, and
$||\ell_{\lambda}||=1$. For this $\ell_\lambda$, there exist
$\rho_i^{(\lambda)},\sigma_j^{(\lambda)}\in\sum\R[\ul t]^2_d$ such
that
\[
\ell_\lambda=\sum_{i=1}^r \rho_i^{(\lambda)} f_i + \sum_{j=1}^s
\sigma_j^{(\lambda)} g_j
\]
and such that
\[
\sigma_j^{(\lambda)}(\gamma(\lambda))=0
\]
for all $j=1,\dots,s$.
\end{itemize}

The statement ($\dagger$) is true, with $d\ge 1$ not depending on
$\lambda$: For $\lambda\in (0,1]$, let $\ell_\lambda\in\R[\ul
t]_1$ be such that $\{\ell_\lambda=0\}$ is a supporting hyperplane
of $S$ passing through $\gamma(\lambda)$, and such that
$||\ell_\lambda||=1$ and $\ell_{\lambda}|_S\ge 0$. By hypothesis,
$\ell_\lambda\in\QM(\{f_i\},\{g_j\})_d$ with $d$ not depending on
$\lambda$, which yields the desired representation. Note that
$\sigma_j^{(\lambda)}(\gamma(\lambda))=0$ is automatic, since
$g_j(\gamma(\lambda))\neq 0$,  but
$\ell_\lambda(\gamma(\lambda))=0$.

Furthermore, because the degree-bound $d$ is
fixed, ($\dagger$) can be expressed as a first-order formula in the
language of ordered rings. Thus ($\dagger$) holds over any real closed
extension field $R$ of $\R$, by the model-completeness of the theory of
real closed fields. Let $R$ be any proper (hence non-archimedean)
extension field and let $\epsilon\in R$, $\epsilon>0$, be an
infinitesimal element with respect to $\R$. We apply $(\dagger)$ with
$\lambda=\epsilon$ and get
\[
\ell_\epsilon=\sum_{i=1}^r \rho_i^{(\epsilon)} f_i + \sum_{j=1}^s
\sigma_j^{(\epsilon)} g_j  \leqno{(\ddagger)}
\]
with
\[
\sigma_j^{(\epsilon)}(\gamma(\epsilon))=0
\]
for all $j=1,\dots,s$. Let $\mathcal{O}$ be the convex hull of
$\R$ in $R$, a valuation ring with maximal ideal $\fm$. Since
$\interior(S)\neq\emptyset$, the quadratic module
$\QM(\{f_i\},\{g_j\})$ has trivial support. As
$||\ell_\epsilon||=1$, it follows that all coefficients of the
polynomials in ($\ddagger$) must lie in $\mathcal{O}$ (see
e.g.~the proof of Lemma 8.2.3 in Prestel and Delzell
\cite{MR1829790}). We can therefore apply the residue map
$\mathcal{O}\into\mathcal{O}/\fm\isom\R$, $a\mapsto \overline{a}$
to the coefficients of ($\ddagger$). From the uniqueness of the
supporting hyperplane $H=\{t_1=0\}$ in $0$ (Step 1), it follows
that $\overline{\ell_\epsilon}=c\cdot t_1$ for some $c\in\R_{>0}$.
This yields the desired expression $(\ast)$.

\medskip

\emph{Step 3.} The existence of ($\ast$) leads to a contradiction:
Substituting $t_1=0$ in ($\ast$) gives
\[
0=\sum_{i=1}^r \rho_i(0,\ul t') f_i(0,\ul t') + \sum_{j=1}^s
\sigma_j(0,\ul t') g_j(0,\ul t')
\]
in $\R[\ul t']$, with $\ul t'=(t_2,\dots,t_n)$. Since all
$f_i(0,\ul t'), g_j(0,\ul t')$ are non-negative on $F_0$, which
has non-empty interior in $H$, it follows that $\rho_i(0,\ul
t')=0$ whenever $f_i(0,\ul t')\neq 0$. In other words, if $t_1$
does not divide $f_i$, then $t_1^2$ divides $\rho_i$ in $\R[\ul
t]$.

Going back to ($\ast$) and substituting $t_2=\cdots=t_n=0$ now gives
\[
t_1=\sum_{i=1}^r \rho_i(t_1,0) f_i(t_1,0) + \sum_{j=1}^s
\sigma_j(t_1,0) g_j(t_1,0)
\]
Since $\sigma_j(0)=0$ for all $j=1,\dots,s$, we now know that
$t_1^2$ divides all terms on the right-hand side, except possibly
$\rho_i(t_1,0)f_i(t_1,0)$ for such $i$ where $t_1|f_i$. In the
latter case, write $f_i=t_1\wt{f}_i$ and note that $\wt{f}_i$
vanishes on $\gamma((0,1])$ since $f_i$ does and $t_1$ does not.
Thus $\wt{f}_i(0)=0$ by continuity which implies
$t_1|\wt{f}_i(t_1,0)$, so $t_1^2|f_i(x_1,0)$ after all. It follows
that $t_1^2$ divides $t_1$, a contradiction.
\end{proof}

\begin{Remarks}
\begin{enumerate}
\item
Note that whether the faces of $S$ are exposed is a purely geometric
condition, independent of the choice of the polynomials $\ul p$. Thus
if $S$ has a non-exposed face, there do not exist polynomials $\ul p$
defining $S$ that yield an exact Lasserre relaxation for $S$.

\item The theorem does \emph{not} imply that a basic closed convex
set with a non-exposed face cannot be semidefinitly representable,
as we will see in the example below. We have only shown that
Lasserre's explicit approach does not work in that case.
\end{enumerate}
\end{Remarks}

\begin{Example}
Consider the basic closed semialgebraic set $S$ defined by
$p_1=t_2-t_1^3$, $p_2=t_1+1$, $p_3=t_2$, $p_4=1-t_2$.
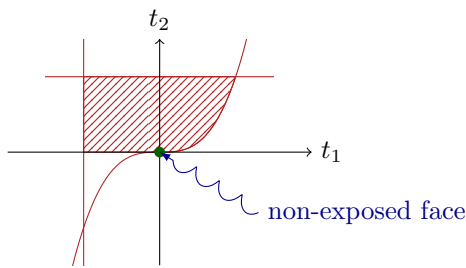
\begin{center}
\begin{tikzpicture}
\begin{scope}
\clip (-1.5,-1.5) rectangle (1.5,1.5);
\pgfsetstrokecolor{FireBrick};
\pgfsetfillpattern{north east lines}{FireBrick};
\draw[smooth,domain=-1.2:1.2] plot({\x},{\x^3});
\draw (-2,1) -- (2,1);
\draw (-1,-1.7) -- (-1,1.7);
\filldraw[smooth,domain=0:1] plot({\x},{\x^3}) -- (-1,1) -- (-1,0) -- (0,0);
\end{scope}
\draw[->] (-2,0) -- (2,0) node[right]{$t_1$};
\draw[->] (0,-1.5) -- (0,1.5) node[above]{$t_2$};
\fill[color=DarkGreen] (0,0) circle (2pt);
\draw[-latex,color=DarkBlue,snake=coil] (1.3,-0.8)
node[right]{non-exposed face} -- (0,0);
\end{tikzpicture}
\end{center}
The point $(0,0)$ is a non-exposed face of $S$ since the only
supporting hyperplane of $S$ passing through $(0,0)$ is the
vertical line $\{t_2=0\}$, whose intersection with $S$ is strictly
bigger than $\{(0,0)\}$. Therefore, there do not exist polynomials
$\ul p$ with $S=\sS(\ul p)$ such that all linear polynomials that
are non-negative on $S$ belong to $\QM(\ul p)_d$ for some
\emph{fixed} value of $d$. On the other hand, the preordering
generated by $p_1,p_2,p_3,p_4$ as above (i.e.~the quadratic module
generated by all products of the $p_i$) contains all polynomials
that are non-negative on $S$. This follows from results of
Scheiderer. Indeed, by the local-global principle \cite[Corollary
2.10]{MR2223624} it suffices to show that the preordering
generated by the $p_i$ is locally saturated. At the origin this
follows from the results in \cite{lpo2} (in particular, Theorem
6.3 and Corollary 6.7). At all other points it follows already
from \cite{MR2223624}, Lemma 3.1.

However, from the result of Helton and Nie, we can deduce that $S$
is in fact semidefinitely representable: For $S$ is the (convex hull
of) the union of the sets $S_1=[-1,0]\times[0,1]$ and
$S_2=\sS(t_2-t_1^3,t_1,1-t_2)$. The set $S_1$ is obviously
semidefinitely representable (even a spectrahedron), while $S_2$
possesses an exact Lasserre-relaxation: More precisely, we claim
that $\QM(t_2-t_1^3,t_1,1-t_2)_3$ contains all linear polynomials
$\ell\in\R[t_1,t_2]$ such that $\ell|_{S_2}\ge 0$. It suffices to show
this for the tangents $\ell_a=t_2-3a^2t_1+2a^3$ to $S_2$ passing
through the points $(a,a^3)$, $a\in [0,1]$ (The claim then follows
from Farkas's lemma). Write
$\ell_a=t_1^3-3a^2t_1+2a^3+(t_2-t_1^3)$. The polynomial
$t_1^3-3a^2t_1+2a^3\in\R[t_1]$ is non-negative on $[0,\infty)$ and
is therefore contained in $\QM(t_1)_3\subseteq\R[t_1]$ (see
Kuhlmann, Marshall, and Schwartz \cite{MR2174483}, Thm.~4.1),
which implies the claim.
\end{Example}

\begin{Remark} \label{clopen} We do not know if the conclusion of Theorem
\ref{main} remains true for ${\rm conv }(S)$ in place of $S$, if
$S$ is not assumed to be convex. It seems unlikely that our proof
can be extended to that case. More generally, is every face of any
Lasserre relaxation  exposed?
\end{Remark}

\noindent
{\bf Note added in proof:} Jo\~ao Gouveia \cite{Joao} showed that our Theorem \ref{main} is optimal in the sense that the questions in Remark \ref{clopen} have negative answers. He also gave an alternative proof of our main theorem which is yet unpublished.


\begin{thebibliography}{10}
\providecommand{\url}[1]{\texttt{#1}}
\providecommand{\urlprefix{}}

\bibitem{MR0445271}
E.~M. Alfsen.
\newblock \emph{Compact convex sets and boundary integrals}.
\newblock Springer-Verlag, New York, 1971.
\newblock Ergebnisse der Mathematik und ihrer Grenzgebiete, Band 57.

\bibitem{Joao}
J.~ao~Gouveia.
\newblock Lasserre relaxations with non-exposed faces.
\newblock Preprint.
\newline\urlprefix\url{http://arxiv.org/abs/0911.2750v1}

\bibitem{MR1940576}
A.~Barvinok.
\newblock \emph{A course in convexity}, vol.~54 of \emph{Graduate Studies in
  Mathematics}.
\newblock American Mathematical Society, Providence, RI, 2002.

\bibitem{MR1659509}
J.~Bochnak, M.~Coste, and M.-F. Roy.
\newblock \emph{Real algebraic geometry}, vol.~36 of \emph{Ergebnisse der
  Mathematik und ihrer Grenzgebiete (3)}.
\newblock Springer-Verlag, Berlin, 1998.

\bibitem{MR1284712}
S.~Boyd, L.~El~Ghaoui, E.~Feron, and V.~Balakrishnan.
\newblock \emph{Linear matrix inequalities in system and control theory},
  vol.~15 of \emph{SIAM Studies in Applied Mathematics}.
\newblock Society for Industrial and Applied Mathematics (SIAM), Philadelphia,
  PA, 1994.

\bibitem{MR2322886}
C.~B. Chua and L.~Tun{\c{c}}el.
\newblock Invariance and efficiency of convex representations.
\newblock \emph{Math. Program.}, \textbf{111~(1-2, Ser. B)}, 113--140, 2008.

\bibitem{HeltonNieNecSuffSDP}
J.~W. Helton and J.~Nie.
\newblock Sufficient and necessary conditions for semidefinite representability
  of convex hulls and sets.
\newblock Preprint.
\newline\urlprefix\url{http://arxiv.org/abs/0709.4017}

\bibitem{MR2533752}
---{}---{}---.
\newblock Semidefinite representation of convex sets.
\newblock \emph{Math. Program.}, \textbf{122~(1, Ser. A)}, 21--64, 2010.

\bibitem{MR2292953}
J.~W. Helton and V.~Vinnikov.
\newblock Linear matrix inequality representation of sets.
\newblock \emph{Comm. Pure Appl. Math.}, \textbf{60~(5)}, 654--674, 2007.

\bibitem{MR1940975}
M.~Kojima and L.~Tun{\c{c}}el.
\newblock On the finite convergence of successive {SDP} relaxation methods.
\newblock \emph{European J. Oper. Res.}, \textbf{143~(2)}, 325--341, 2002.
\newblock Interior point methods (Budapest, 2000).

\bibitem{MR2174483}
S.~Kuhlmann, M.~Marshall, and N.~Schwartz.
\newblock Positivity, sums of squares and the multi-dimensional moment problem.
  {II}.
\newblock \emph{Adv. Geom.}, \textbf{5~(4)}, 583--606, 2005.

\bibitem{MR2505746}
J.~B. Lasserre.
\newblock Convex sets with semidefinite representation.
\newblock \emph{Math. Program.}, \textbf{120~(2, Ser. A)}, 457--477, 2009.

\bibitem{MR2011395}
M.~Marshall.
\newblock Optimization of polynomial functions.
\newblock \emph{Canad. Math. Bull.}, \textbf{46~(4)}, 575--587, 2003.

\bibitem{MR2383959}
---{}---{}---.
\newblock \emph{Positive polynomials and sums of squares}, vol. 146 of
  \emph{Mathematical Surveys and Monographs}.
\newblock American Mathematical Society, Providence, RI, 2008.

\bibitem{Tim}
T.~Netzer.
\newblock On semidefinite representations of sets.
\newblock Preprint.
\newline\urlprefix\url{http://arxiv.org/abs/0907.2764}

\bibitem{TimRainer}
T.~Netzer and R.~Sinn.
\newblock A note on the convex hull of finitely many projections of
  spectrahedra.
\newblock Preprint.
\newline\urlprefix\url{http://arxiv.org/abs/0908.3386}

\bibitem{MR1823953}
V.~Powers and C.~Scheiderer.
\newblock The moment problem for non-compact semialgebraic sets.
\newblock \emph{Adv. Geom.}, \textbf{1~(1)}, 71--88, 2001.

\bibitem{MR1829790}
A.~Prestel and C.~N. Delzell.
\newblock \emph{Positive polynomials}.
\newblock Springer Monographs in Mathematics. Springer-Verlag, Berlin, 2001.

\bibitem{MR1342934}
M.~Ramana and A.~J. Goldman.
\newblock Some geometric results in semidefinite programming.
\newblock \emph{J. Global Optim.}, \textbf{7~(1)}, 33--50, 1995.

\bibitem{MR2198215}
J.~Renegar.
\newblock Hyperbolic programs, and their derivative relaxations.
\newblock \emph{Found. Comput. Math.}, \textbf{6~(1)}, 59--79, 2006.

\bibitem{lpo2}
C.~Scheiderer.
\newblock Weighted sums of squares in local rings and their completions, ii.
\newblock Preprint.
\newline\urlprefix\url{http://www.maths.manchester.ac.uk/raag/}

\bibitem{MR2223624}
---{}---{}---.
\newblock Sums of squares on real algebraic surfaces.
\newblock \emph{Manuscripta Math.}, \textbf{119~(4)}, 395--410, 2006.

\bibitem{MR1379041}
L.~Vandenberghe and S.~Boyd.
\newblock Semidefinite programming.
\newblock \emph{SIAM Rev.}, \textbf{38~(1)}, 49--95, 1996.

\end{thebibliography}
\end{document}